\documentclass{amsart}
\usepackage[utf8]{inputenc}
\usepackage{graphicx}
\usepackage{amsmath, amssymb, amsthm}
\usepackage[margin=1in]{geometry}
\usepackage{tikz}
\usepackage{multicol}
\usepackage{esint}
\usepackage{hyperref}
\usepackage{stmaryrd}
\usepackage{dirtytalk}
\usepackage{pgfplots}
\pgfplotsset{compat=1.18} 
\allowdisplaybreaks
\hypersetup{
    colorlinks,
    citecolor=black,
    filecolor=black,
    linkcolor=black,
    urlcolor=black
}
\usetikzlibrary{arrows}
\newtheorem{theorem}{Theorem}

\newtheorem{lemma}{Lemma}
\newtheorem{remark}{Remark}
\newtheorem{corollary}{Corollary}
\newtheorem{proposition}{Proposition}

\newtheorem{example}{Example}

\newcommand{\NN}{\mathbb{N}}

\newcommand{\PP}{\mathbb{P}}

\newcommand{\RR}{\mathbb{R}}

\newcommand{\TT}{\mathbb{T}}

\newcommand{\ZZ}{\mathbb{Z}}

\newcommand{\B}{\mathcal{B}}
\newcommand{\C}{\mathcal{C}}
\newcommand{\D}{\mathcal{D}}

\newcommand{\I}{\mathcal{I}}

\renewcommand{\O}{\mathcal{O}}

\newcommand{\Z}{\mathcal{Z}}

\title[PLE of Cat Map]{Uniform Positivity of the Lyapunov Exponent for $C^1$ Monotone potentials generated by the Cat Map}
\author{Nicholas Chiem}
\address{Department of Mathematics, University of California Riverside, Riverside, CA-92521, USA}
\email{nchie005@ucr.edu}

\begin{document}

\begin{abstract}
    We consider an Arnold's Cat Map generated $C^1$ bounded potential with the directional derivative in the unstable direction bounded away from zero. We show that the Lyapunov exponent for the associated Shr\"odinger Operator is uniformly positive for all energies provided the coupling is sufficiently large. 
\end{abstract}

\maketitle

\section{Introduction}
We let $\TT^2 = \RR^2/\ZZ^2$. Define a linear map $T:\RR^2 \rightarrow \RR^2$ 
\[T(x,y) = (2x+y,x+y)\]
which, abusing notation, induces a map $T:\TT^2 \rightarrow \TT^2$
\[\omega :=(\omega_1,\omega_2) \mapsto (2\omega_1 + \omega_2, \omega_1+ \omega_2) \mod \ZZ.\]
We equip $\TT^2$ with Borel $\sigma$-algebra $\B$ and measure $\mu = m^2$ where $m= Leb$. The family of Schr\"odinger Operators is defined as $H_{\lambda, v, \omega}: \ell^2(\ZZ) \rightarrow \ell^2(\ZZ)$
\begin{equation}\label{d.sop}
[H_{\lambda,v,\omega}\psi]_n = \psi_{n+1} + \psi_{n-1} + \lambda v(T^n\omega)\psi_n
\end{equation}
where $\lambda\in \RR$ is the \textbf{coupling constant}, $v:\TT^2 \rightarrow \RR$ is a measurable \textbf{potential}, and $\omega\in \TT^2$ is the \textbf{phase}. 

We study the eigenvalue equation $H_{\lambda,v,\omega}\psi = E\psi$, which admits a map $A^{(E-\lambda v)}: \TT^2 \rightarrow \mathrm{SL}(2,\RR)$
\begin{equation}\label{d.sc}
    A^{(E-\lambda v)}(\omega) = \begin{pmatrix}
    E - \lambda v(\omega) & -1\\
    1 &0
\end{pmatrix}
\end{equation}
which is known as a \textbf{cocycle} map. We can then define a family of dynamical systems
\begin{align*}
    (T, A): \TT^2 \times \RR^2 &\rightarrow \TT^2 \times \RR^2\\
    (\omega, \vec{v}) &\mapsto (T\omega, A^{(E-\lambda v)}(\omega) \vec{v}) 
\end{align*}
called \textbf{Schr\"odinger cocycles}. We write iterates of the map as $(T,A^{(E-\lambda v)})^n = (T^n, A^{(E-\lambda v)}_n)$, where $A^{(E-\lambda v)}_n$ is defined as
\[A^{(E-\lambda v)}_n(\omega) = \begin{cases}
    A^{(E-\lambda v)}(T^{n-1}\omega) \cdots A^{(E-\lambda v)}(\omega) & n > 0\\
    I_2 & n = 0\\
    A^{(E-\lambda v)}(T^{n}\omega)^{-1} \cdots A^{(E-\lambda v)}(T\omega)^{-1} & n< 0
\end{cases}\]
and referred to as the \textbf{n-step transfer matrix}. The eigenvalue equation is directly connected to cocycles, as
\[H_{\lambda,v,\omega}\psi = E\psi \iff \begin{pmatrix}
    \psi_{n}\\
    \psi_{n-1}
\end{pmatrix} = A^{(E-\lambda v)}_n \begin{pmatrix}
    \psi_0\\
    \psi_{-1}
\end{pmatrix} \text{ for all } n\in \ZZ\]
for any sequence $(\psi_n)_{n\in \ZZ}$ and not necessarily in $\ell^2(\ZZ)$. The \textbf{Lyapunov exponent} is defined as
\begin{align}
    L(E; \lambda) &= \lim_{n\rightarrow \infty} \frac{1}{n}\int_{\TT^2} \log||A^{(E-\lambda v)}_n(\omega)|| d\mu(\omega)\\
    &= \inf_{n\geq 1} \frac{1}{n}\int_{\TT^2} \log||A^{(E-\lambda v)}_n(\omega)|| d\mu(\omega) \nonumber \\
    &\geq 0 \nonumber
\end{align}
where the existence of the limit follows by subadditivity. Moreover, by Kingman's Subadditive Ergodic Theorem we find
\[L(E;\lambda) = \lim_{n\rightarrow \infty}\frac{1}{n}\log||A^{(E-\lambda v)}_n(\omega)|| d\mu(\omega) \text{ for a.e } \omega\in \TT^2.\]

Recall for the Cat Map that
\[\alpha = \frac{3+\sqrt{5}}{2} \text{ and } \vec{u} = \begin{pmatrix}
        \frac{1+\sqrt{5}}{2}\\
        1
\end{pmatrix}\]
is the largest eigenvalue and the corresponding eigenvector or unstable vector, respectively. We recall that an unstable foliation of the Cat Map has dense leaves in $\TT^2$. For our purposes, we cut the dense leaves into line segments. More precisely, for $\omega\in \TT^2$ one can define a unique line segment in the unstable direction containing $\omega$. This line segment is defined to begin and end on the boundary of the fundamental domain and any element on the interior of the line segment is not a boundary point of the fundamental domain. Moreover, the line segment only includes the starting boundary point and not the ending boundary. We call the line segments \textbf{local unstable leaves} and can parametrize the family of local unstable leaves by projecting each local unstable leaf to the diagonal of the fundamental domain. That is, for each $z\in (0,\sqrt{2})$ we denote $W_z$ as a local unstable leaf. Geometrically, one has: 
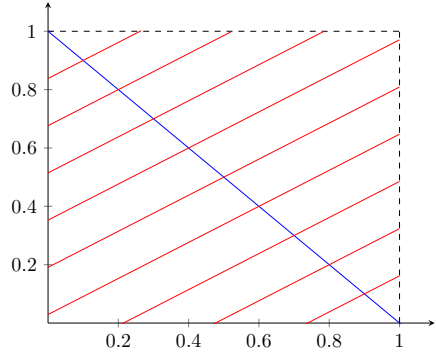
\begin{center}
    \begin{tikzpicture}[scale=0.75]
    \begin{axis}[xmin=0, xmax=1.1, ymin=0, ymax=1.1, axis x line=middle, axis y line=middle]
        \draw[dashed] (1,0) -- (1,1);
        \draw[dashed] (0,1) -- (1,1);
        \addplot[blue]{1-x};
        \addplot[red, smooth,domain=0:.265] {(sqrt(5)-1)*.5*x + 1 - (1+sqrt(5))*.5*.1};
        \addplot[red, smooth,domain=0:.52] {(sqrt(5)-1)*.5*x + 1 - (1+sqrt(5))*.5*.2};
        \addplot[red, smooth,domain=0:.785] {(sqrt(5)-1)*.5*x + 1 - (1+sqrt(5))*.5*.3};
        \addplot[red, smooth,domain=0:1] {(sqrt(5)-1)*.5*x + 1 - (1+sqrt(5))*.5*.4};
        \addplot[red, smooth,domain=0:1] {(sqrt(5)-1)*.5*x + 1 - (1+sqrt(5))*.5*.5};
        \addplot[red, smooth,domain=0:1] {(sqrt(5)-1)*.5*x + 1 - (1+sqrt(5))*.5*.6};
        \addplot[red, smooth,domain=0:1] {(sqrt(5)-1)*.5*x + 1 - (1+sqrt(5))*.5*.7};
        \addplot[red, smooth,domain=0:1] {(sqrt(5)-1)*.5*x + 1 - (1+sqrt(5))*.5*.8};
        \addplot[red, smooth,domain=0:1] {(sqrt(5)-1)*.5*x + 1 - (1+sqrt(5))*.5*.9};
    \end{axis}
\end{tikzpicture}
\end{center}
Each red line represents a local unstable leaf and the blue line is the diagonal that gives a natural parametrization of $W_z$. Let $\ell_z$ denote be the length of $W_z$, then a direct computation shows that $\ell_z\in [0, \sqrt{\frac{5-\sqrt{5}}{2}}]$ and $W_z$ is homeomorphic to $[0,\ell_z)$.

\section{Main Result}
Throughout the paper we let $c,C$ represent universal constants where $c$ is some small constant and $C$ is some large constant. These constants only depend on the bounded measurable potential $v:\TT^2\rightarrow \RR$ that satisfies
\begin{enumerate}
    \item[(i)] $v\in C^1([0,1)^2, \RR)$.
    \item[(ii)] For all $\omega\in \TT^2$, then the directional derivative of $v$ in the unstable direction is uniformly away from zero. That is, $D_{\vec{u}}v(\omega) > c$. 
\end{enumerate}
We provide some enjoyable examples of such potentials. 
\begin{example}
    Fix an $n\in \NN$ and define a polynomial in two variables
    \begin{align*}
        v: [0,1)^2 &\rightarrow \RR\\
        (\omega_1,\omega_2) &\mapsto \sum_{k=0}^na_k \omega_1^k + b_k \omega_2^k
    \end{align*}
    where $a_k,b_k \in \RR\geq 0$ and $a_1 \neq 0$ or $b_1\neq 0$. Then $v\in C^1([0,1)^2)$ and bounded above by $\sum_{k=0}^n a_k + b_k$. A direct computation gives the Jacobian of $v$:
    \[J_v = \begin{pmatrix}
        \sum_{k=1}^{n} ka_k \omega_1^{k-1}& \sum_{k=1}^{n} kb_k\omega_2^{k-1} 
    \end{pmatrix}\]
    Then the directional derivative with respect to the unstable direction is:
    \[D_{\vec{u}}v(\omega) = J_v\vec{u} = u_1\sum_{k=1}^{n} ka_k \omega_1^{k-1} + u_2 \sum_{k=1}^{n} kb_k \omega_2^{k-1} \geq u_1a_1 + u_2b_1\]
    Therefore for any $\omega \in [0,1)^2$, we may conclude $D_{\vec{u}}v(\omega) > c$.
\end{example}

\begin{example}
    Consider an exponential function defined as
    \begin{align*}
        v:[0,1)^2 &\rightarrow \RR\\
        (\omega_1, \omega_2) &\mapsto e^{\omega_1+ \omega_2}
    \end{align*}
    which is $C^1([0,1)^2)$ and bounded by $e^2$. The Jacobian becomes
    \[J_v = \begin{pmatrix}
        e^{\omega_1+ \omega_2} & e^{\omega_1+ \omega_2}
    \end{pmatrix}.\]
    While the directional derivative in the unstable direction 
    \[D_{\vec{u}}v(\omega) = (u_1+u_2)e^{\omega_1+ \omega_2} \geq u_1 + u_2 > c\]
    which holds for any $\omega\in [0,1)^2$. 
\end{example}

\begin{example}
    Consider the logarithm function defined as
    \begin{align*}
        v:[0,1)^2 &\rightarrow \RR\\
        (\omega_1, \omega_2) &\mapsto \log(1+\omega_1+\omega_2)
    \end{align*}
    which is $C^1([0,1)^2)$ and bounded by $\log(3)$. Computing we find
    \[J_v = \begin{pmatrix}
        \frac{1}{1+\omega_1+\omega_2} & \frac{1}{1+\omega_1+\omega_2}
    \end{pmatrix}\]
    and it follows that $D_{\vec{u}}v(\omega) \geq \frac{1}{3}(u_1 + u_2) > c$.
\end{example}

From now on, WLOG, we will assume that $||v||_\infty \leq 1$ and remark that monotonicity implies discontinuities will occur on the boundary of the fundamental domain. 
\begin{theorem}\label{t.PLE}
    Let $v$ be defined as above and consider the family of Schr\"odinger operators as in (\ref{d.sop}). Then there exists a $C_0 = C_0(v)>0$ such that for any $\lambda>0$ we have
    \[L(E;\lambda) > \log\lambda - C_0 \text{ for all } E\in \RR.\]
\end{theorem}

With $\lambda>0$ sufficiently large we can guarantee uniform positivity of the Lyapunov exponent, which is a signature of chaos in a dynamical system. 
Moreover, the Lyapunov exponent plays a key role in spectral analysis of ergodic one-dimensional Schr\"odinger operators, see \cite{wilkinson}. In particular, having positive Lyapunov exponent can be a strong indication of \textbf{Anderson Localization}. That is, the family of operators has pure point spectrum with exponentially decaying eigenfunctions. This phenomenon is known to occur for potentials that are sufficiently random and a well studied case is when the potential is generated by iid random variables. In general, the question of being able to achieve localization with less random base dynamics is difficult.

One of the difficulties of achieving localization is being able to achieve uniformly positive Lyapunov exponents. Showing uniform positivity depends on the dynamics and two well known cases lie in different degrees of randomness. The most random case is when the potentials are generated by iid random variables. In this case, uniform positivity was established by Furstenberg \cite{fursten} for all nonzero couplings. On the other hand, for quasi-periodic potentials, it has been shown that a large class of potentials can achieve uniform positivity for large couplings. 

For intermediate randomness much less is known. Our focus is on hyperbolic dynamics. For small coupling constants, \cite{ChuSpe} showed positivity away from the center and edges of the spectrum for deterministic potentials using methods from \cite{FigPas}. In the same spectral regime, \cite{BouSch} were able to establish a large deviation theorem using methods from \cite{FigPas}. This allowed them to achieve Anderson localization for almost every phase. On the other hand, for sufficiently large couplings and sufficiently large hyperbolicity, \cite{Bje} established positivity for circle endomorphisms and \cite{ZhaLi} showed positivity for all energies on expanding toral automorphisms. Moreover, \cite{Bou2} established positivity for any $\lambda\neq 0$, as long as one has sufficient hyperbolicity and $C^1$ non-constant potentials. Positivity for hyperbolic base dynamics was also shown by \cite{SS}, \cite{DK}, \cite{K}. Recently, a more complete answer for hyperbolic base dynamics was given from \cite{AvDaZh}. The group is able to show that for $\alpha$-Holder continuous functions, then the set $\Z = \{E\in \RR: L(E) = 0\}$ is discrete. Moreover, there exists a constant $\lambda_0(v)$ so that $\Z$ if finite for all $0 < \lambda < \lambda_0$. Lastly, there is an open dense set $\O^\alpha\subset C^\alpha(\TT^d,\RR)$ such one has uniformly positive Lyapunov exponents for all potentials in $\O^\alpha$. For more general information on results of positive Lyapunov exponents we refer to \cite{WanZha}. 

Returning to the context of our problem, Young \cite{young} remarks that if one has a cocyle of the form
\begin{equation}\label{d.ycoc}
    A_\epsilon(x,y) = \begin{pmatrix}
    \lambda & 0\\
    0 & \lambda^{-1}
    \end{pmatrix}R_{\phi_\epsilon(x,y)}
\end{equation}
then one has positive Lyapunov exponents as long as $\phi_\epsilon$ satisfies three conditions. Letting $J_\epsilon\subset S^1$ compact, then
\begin{enumerate}
    \item[(a)] $\phi_\epsilon \equiv 0$ outside of $J_\epsilon \times S^1$.
    \item[(b)] On $J_\epsilon \times S^1$, $\phi_\epsilon$ increases monotonically from $0$ to $2\pi$ on the leaves of an unstable foliation.
    \item[(c)] On $\phi_\epsilon^{-1}[\beta, 2\pi - \beta]$, the directional derivatives along the leaves of an unstable foliation are $\geq \epsilon^{-1}$.
\end{enumerate}

Our paper, inspired by \cite{zhenghe2}, focuses on Shr\"odinger Cocycles which can be reduced to a similar form of (\ref{d.ycoc}). The difference from (\ref{d.ycoc}) to our setting is that the matrix entries $\lambda$ are no longer constants, but rather functions. This can be dangerous as this may cancel the expanding effects of the base dynamics. We remark that this cancellation is where sufficient hyperbolicity is useful for \cite{ZhaLi}, \cite{Bou2}, \cite{Bje}. 

In our setting, this becomes a battle between the hyperbolic base dynamics represented by eigenvalue $\alpha$ and the hyperbolicity of the cocycle represented by the sufficiently large coupling constant $\lambda$. We are able to control this battle with the key assumption that the function attached to the rotation matrix has directional derivative, in the unstable direction, uniformly bounded away from zero. Moreover, this assumption generalizes the function attached to the rotation matrix.

The core idea for us to show positive Lyapunov exponents is to reduce the problem to the local unstable leaves. This is done through an application of Fubini's theorem, which allows us to integrate over $z\in(0,\sqrt{2})$ and the $W_z$. After reduction, we use the fact that each $W_z$ has the property that every $T$-action on $W_z$ moves the unstable local leaf to another unstable local leaf and stretches the length by a factor of $\alpha$. This stretching will never form a closed loop, which is a consequence of the irrational slope. This implies that the image of $W_z$ under $T^n$ becomes a subset of a union of $W_j$ for $j\in \Lambda$ for an indexing set $\Lambda$. The number of elements in the union is completely dependent on the length of $T^nW_z$. This reduces the 2-dimensional problem to 1-dimension, which allows us to use arguments similar to \cite{zhenghe2}.  

Moreover, the argument provided for the Cat Map works for any hyperbolic $\mathrm{SL}(2,\ZZ)$ matrix. Let $T$ represent any hyperbolic matrix action on $\TT^2$ and let $\beta>1$ represent the largest eigenvalue. There are two cases that occur for the unstable leaves, which depend on whether the associated eigenvector produces an irrational or rational slope. The first case is identical to the Cat Map where the only difference is swapping $\alpha$ with $\beta$ for all proofs. The second case is when the unstable leaves become copies of $\RR/\ZZ$. This means that the action of $T$ on the unstable leaves behaves exactly as the map $x\mapsto \beta x$ on $\RR/\ZZ$. Therefore, by the argument in \cite{zhenghe2} and replacing $2$ with $\beta$ gives positivity Lyapunov exponents on unstable leaves. We formalize this exposition in a theorem.
\begin{theorem}\label{t.hple}
    Consider any hyperbolic matrix in $\mathrm{SL}(2,\ZZ)$ with eigenvalues $\beta, \beta^{-1}$ where $\beta> 1$. If $v$ is defined as in Theorem \ref{t.PLE}, then there exists a $C_0 = C_0(v)>0$ such that for any $\lambda > 0$ we have
    \[L(E;\lambda) > \log \lambda - C_0 \text{ for all } E\in \RR.\]
\end{theorem}

\section{Proof of Theorems}
Before we are able to show positive Lyapunov exponents we must move to a cocycle of similar form to (\ref{d.ycoc}). This is done by polar decomposition of the Schr\"odinger Cocycle and details are written in Appendix \ref{apd}. To summarize, we 
let $t = \frac{E}{\lambda}$ and $t\in \I = [-1, 2]$. Define functions $r,g: \TT^2\times \I \rightarrow \RR$ as
\begin{align*}
    r(\omega, t) &= t - v(\omega)\\
    g(\omega, t) &= r^2(\omega) + 1
\end{align*}
By the assumptions on $v$ we have
\[c < g(\omega,t) < C \text{ for all } \TT^2\times \I\]
Polar decomposition of $A^{(E-\lambda v)}$ allows us to define
\begin{equation}\label{e.coc}
    A(\omega, t, \lambda):= \Lambda(T\omega)O(\omega) = \begin{pmatrix}
    \lambda \sqrt{g(T\omega, t)} & 0 \\
    0 & \frac{1}{\lambda} \sqrt{\frac{1}{g(T\omega, t)}}
\end{pmatrix}R_{\theta(\omega, t)}
\end{equation}
where 
\[R_\gamma = \begin{pmatrix}
    \cos\gamma & -\sin\gamma\\
    \sin\gamma & \cos\gamma
\end{pmatrix}\]
and $\theta(\omega,t)$ is 
\[\theta(\omega, t) = \cot^{-1}(t-v(\omega)).\]
Let $\theta_0 = \theta(\omega,t)$ and for every $n\in \NN$ we define
\[\phi_n(\omega, t) = \cot^{-1}(\lambda^2 g(T^n\omega)\theta_{n-1}) \text{ and } \theta_n(\omega, t) = \phi_n(\omega, t) + \theta(T^n\omega).\]

We denote the Lyapunov exponent for the cocycle $(T, A(\cdot, t, \lambda))$ as $L(t;\lambda)$. By an argument in \cite{zhenghe1}, it suffices to show the following Theorem \ref{t.plepd} to prove Theorem \ref{t.PLE}.
\begin{theorem}\label{t.plepd}
    For $v$ as in Theorem \ref{t.PLE}, then there exists a $C_0 = C_0(v)>0$ such that for any $\lambda>0$ we have
    \[L(t;\lambda) > \log\lambda - C_0 \text{ for all } t\in \I.\]
\end{theorem}
\noindent From this point on we suppress the notation of $t$, as all estimates will be uniform in $t$. 

One can reduce the problem to show positive Lyapunov exponents on unstable local leaves. Observe that for $g\in L^1(\mu)$:
\begin{equation}\label{fubini}
    \int_{\TT^2} g(\omega) d\mu(\omega) = \int_{(0,\sqrt{2})}\int_{W_z}g(\omega) dm(\omega) dm(z),
\end{equation}
which is a direct consequence of \cite[Theorem 2.44]{folland} and Fubini. The purpose of (\ref{fubini}) is that when applied on $\frac{1}{n}\log||A_n(\omega)||$, we find
\[\frac{1}{n}\int_{\TT^2} \log||A_n(\omega)|| d\mu(\omega) = \frac{1}{n}\int_{(0,\sqrt{2})} \int_{W_z} \log||A_n(\omega)|| dm(\omega)dm(z)\]
for every $n\in \NN$. Let us denote the Lyapunov exponent on unstable local leaves as $L(t;\lambda, z)$.
\begin{proposition}\label{p.fPLE}
    There exists a $C_0 = C_0(v)>0$ such that for any $\lambda > 0, z\in (0,\sqrt{2})$
    \[L(t;\lambda, z)> \log \lambda - C_0\]
    for all $t\in \I$.
\end{proposition}
With this proposition we can show theorem \ref{t.plepd}.

\begin{proof}[Proof of Theorem \ref{t.plepd}.]
    For $n\in \NN, z\in (0,\sqrt{2})$, define $g_n(\omega) = \frac{1}{n}\log||A_n(\omega)||$ and $f_n(z) = \int_{W_z} g_n(\omega) dm(\omega)$. Notice that
    \[||g_n||_\infty < \log ||A||\]
    and, by definition of $L(t;\lambda, z)$, $f_n(z) \rightarrow L(t;\lambda, z)$ as $n\rightarrow \infty$. Observe that $m(W_z) < \sqrt{2}$ and
    \[||f_n||_{\infty} < \sqrt{2}\log ||A|| \text{ for all } n\in \NN.\]
    The constant function $\sqrt{2}\log||A|| \in L^1(\mu)$, so we may apply Lebesgue Dominated Convergence Theorem. Putting everything together,
    \begin{align*}
        L(t;\lambda) &= \lim_{n\rightarrow \infty} \int_{\TT^2} g_n(\omega) d\mu\\
        &= \lim_{n\rightarrow \infty} \int_{\TT}f_n(z) dm(z)\\
        &= \int_{\TT} \lim_{n\rightarrow \infty} f_n(z) dm(z)\\
        &= \int_{\TT} L(t; \lambda, z) dm(z)\\
        &> \int_{\TT} \log\lambda - C_0 dm(z)\\
        &= \log\lambda - C_0.
    \end{align*}
\end{proof}

To show positive Lyapunov exponents, we use the methods in \cite{zhenghe2} to show on each $W_z$ one has positive Lyapunov exponents. To begin we must control the directional derivative of $D_{\vec{u}}\theta_n$ from below. 

\begin{lemma}\label{l.lb}
    For any $t\in \I$ 
    \begin{equation}\label{i.lb}
        D_{\vec{u}}\theta_n(\omega) > c\alpha^n
    \end{equation}
    for all $n\in \ZZ_{\geq0}$ and $\omega \in \TT^2$ where $\theta_n$ is differentiable.
\end{lemma}
\begin{proof}
    Consider $n=0$ and let $\omega \in \TT^2$.
    \begin{align*}
        D_{\vec{u}}\theta_0(\omega) &= J_{\theta}(\omega)\vec{u}\\
        &= \begin{pmatrix}
            \partial_{\omega_1} \theta(\omega) & \partial_{\omega_2}\theta(\omega)
        \end{pmatrix} \vec{u}\\
        &= 
        \begin{pmatrix}
            \frac{\partial_{\omega_1}v(\omega)}{g(\omega)} & \frac{\partial_{\omega_2} v(\omega)}{g(\omega)}
        \end{pmatrix} \vec{u}\\
        &= \frac{1}{g(\omega)} D_{\vec{u}}v(\omega)\\
        &> c
    \end{align*}
    Assume for induction, (\ref{i.lb}) holds for some $n-1$, then we show that (\ref{i.lb}) holds for $n$. A minor computation gives,
    \begin{align*}
        J_{\theta_n}(\omega) &= J_{\phi_n}(\omega) + J_{\theta\circ T^n}(\omega)\\
        &= \begin{pmatrix}
            \partial_{\omega_1}\phi_n(\omega) & \partial_{\omega_2}\phi_n(\omega)
        \end{pmatrix} + 
        \begin{pmatrix}
            \partial_{\omega_1}[\theta(T^n\omega)] & \partial_{\omega_2}[\theta(T^n\omega)]
        \end{pmatrix}\\
        &= \begin{pmatrix}
            \partial_{\omega_1}\phi_n(\omega) + \partial_{\omega_1}[\theta(T^n\omega)]& \partial_{\omega_2}\phi_n(\omega) + \partial_{\omega_2}[\theta(T^n\omega)]
        \end{pmatrix}
    \end{align*}
    which implies
    \[D_{\vec{u}}\theta_n(\omega) = \alpha \partial_{\omega_1}\phi_n(\omega) + \alpha \partial_{\omega_1}[\theta(T^n\omega)] + \partial_{\omega_2}\phi_n(\omega) + \partial_{\omega_2}[\theta(T^n\omega)].\]
    We will compute $\partial_{\omega_1}$, as the computation for $\partial_{\omega_2}$ is completely analogous.
    \begin{align}
        \partial_{\omega_1}[\phi_n(\omega)] &= \frac{-\lambda^2}{1+(\lambda^2g(T^n\omega)\cot\theta_{n-1}(\omega))^2}\Bigg[\partial_{\omega_1}[g(T^n\omega)]\cot\theta_{n-1}(\omega) + g(T^n\omega)\partial_{\omega_1}[\cot\theta_{n-1}(\omega)] \Bigg] \nonumber \\
        &= -\frac{\lambda^2\partial_{\omega_1}[g(T^n\omega)]\cot\theta_{n-1}(\omega)}{1+(\lambda^2g(T^n\omega)\cot\theta_{n-1}(\omega))^2} - \frac{\lambda^2g(T^n\omega)\partial_{\omega_1}[\cot\theta_{n-1}(\omega)]}{1+(\lambda^2g(T^n\omega)\cot\theta_{n-1}(\omega))^2} \nonumber\\
        &= \frac{2\lambda^2\cot\theta_{n-1}(\omega)(t-v(T^n\omega))}{1+(\lambda^2g(T^n\omega)\cot\theta_{n-1}(\omega))^2}\partial_{\omega_1}[v(T^n\omega)] \label{x1}\\
        &+ \frac{\lambda^2g(T^n\omega)(1+\cot^2\theta_{n-1}(\omega))}{1+(\lambda^2g(T^n\omega)\cot\theta_{n-1}(\omega))^2}\partial_{\omega_1}[\theta_{n-1}(\omega)] \label{x2}\\
        \partial_{\omega_1}[\theta(T^n\omega)] &= \partial_{\omega_1}[\cot^{-1}(t-v(T^n\omega)] \nonumber\\
        &= \frac{1}{1 + (t-v(T^n\omega))^2}\partial_{\omega_1}[v(T^n\omega)]\nonumber\\
        &= \frac{1}{g(T^n\omega)}\partial_{\omega_1}[v(T^n\omega)]\label{x3}
    \end{align}
    We consider
    \[(\ref{x1}) + (\ref{x3}) =  \Bigg[\frac{1}{g(T^n\omega)} + \frac{2\lambda^2\cot\theta_{n-1}(\omega)(t-v(T^n\omega))}{1+(\lambda^2g(T^n\omega)\cot\theta_{n-1}(\omega))^2}\Bigg] \partial_{\omega_1}[v(T^n\omega)]\]
    and let $h(\omega) = \lambda^2g(T^n\omega)\cot\theta_{n-1}(\omega)$ and let $f(\omega)$ denote the inside of the brackets. Then for any $\omega\in \TT^2$
    \begin{align}
        |f(\omega)| &< \Bigg|\frac{1}{g(T^n\omega)}\Bigg| + \Bigg|\frac{\lambda^2\cot\theta_{n-1}(\omega)g(T^n\omega)2(t-v(T^n\omega))}{1+(\lambda^2g(T^n\omega)\cot\theta_{n-1}(\omega))^2}\Bigg| \label{ineqf<}\\
        &< 1 + \frac{4|h(\omega)|}{1+h^2(\omega)}\nonumber\\
        &< 3 \nonumber
    \end{align}
    If $\omega\in \TT^2$ satisfies $|g(T^n\omega)\cot\theta_{n-1}(\omega)|<C\lambda^{-\frac{3}{2}}$ then
    \begin{equation}\label{ineq<C}
        \frac{\lambda^2g(T^n\omega)(1+\cot^2\theta_{n-1}(\omega))}{1+(\lambda^2g(T^n\omega)\cot\theta_{n-1}(\omega))^2}\partial_{\omega_1}[\theta_{n-1}] > c\lambda\partial_{\omega_1}[\theta_{n-1}(\omega)]
    \end{equation}
    On the other hand, if $\omega \in \TT^2$ satisfies $|g(T^n\omega)\cot\theta_{n-1}(\omega)|>c\lambda^{-\frac{3}{2}}$, then
    \begin{align}
        f(x) &> c - \frac{\lambda^2\cot\theta_{n-1}(\omega)g(T^n\omega)2(t-v(T^n\omega))}{1+(\lambda^2g(T^n\omega)\cot\theta_{n-1}(\omega))^2} \label{ineqf>}\\
        &> c - \frac{C|h(\omega)|}{1+h^2(\omega)} \nonumber\\
        &> c - \frac{C\lambda^{\frac{1}{2}}}{1+\lambda} \nonumber\\
        &>c. \nonumber
    \end{align}
    Note that for any $\omega\in \TT^2$ we have
    \begin{equation}\label{ineqany}
        \frac{\lambda^2g(T^n\omega)(1+\cot^2\theta_{n-1}(\omega))}{1+(\lambda^2g(T^n\omega)\cot\theta_{n-1}(\omega))^2}\partial_{\omega_1}[\theta_{n-1}(\omega)] > 0
    \end{equation}
    If $|g(T^n\omega)\cot\theta_{n-1}|<C\lambda^{-\frac{3}{2}}$, then by (\ref{ineqf<}) and (\ref{ineq<C})
    \begin{equation}\label{px1}
        \partial_{\omega_1}[\theta_n(\omega)] > c\lambda \partial_{\omega_1}[\theta_{n-1}(\omega)] - 3\partial_{\omega_1}[v(T^n\omega)]
    \end{equation}
    If $|g(T^n\omega)\cot\theta_{n-1}|>c\lambda^{-\frac{3}{2}}$, then by (\ref{ineqf>}) and (\ref{ineqany})
    \begin{equation}\label{px2}
        \partial_{\omega_1}[\theta_{n}(\omega)] > f(\omega) \partial_{\omega_1}[v(T^n\omega)] >c \partial_{\omega_1}[v(T^n\omega)]
    \end{equation}
    By the same computation, (\ref{px1}) and (\ref{px2}) hold for $\partial_{\omega_2}$. Consider the case when $|g(T^n\omega)\cot\theta_{n-1}|<C\lambda^{-\frac{3}{2}}$ we find
    \begin{align*}
        D_{\vec{u}}\theta_n &=  u_1\partial_{\omega_1}\theta_n + \partial_{\omega_2}\theta_n\\
        &> c\lambda(u_1\partial_{\omega_1}[\theta_{n-1}] + \partial_{\omega_2}[\theta_{n-1}]) - 3(u_1\partial_{\omega_1}[v(T^n\omega)] + \partial_{\omega_2}[v(T^n\omega)])\\
        &= c\lambda D_{\vec{u}}[\theta_{n-1}] - 3D_{\vec{u}}[v\circ T^n](\omega)\\
        &> c\lambda \alpha^{n-1} - 3\alpha^n  D_{\vec{u}}[v](T^n\omega)\\
        &> c\alpha^n.
    \end{align*}
    Where we use the fact that $\lambda$ is sufficiently large. For the latter case, $|g(T^n\omega)\cot\theta_{n-1}|>c\lambda^{-\frac{3}{2}}$ we have
    \begin{equation}\label{i.lb2}
        D_{\vec{u}}\theta_n = u_1\partial_{\omega_1}\theta_n + \partial_{\omega_2}\theta_n > cu_1\partial_{\omega_1}[v(T^n\omega)] + c\partial_{\omega_2}[v(T^n\omega)] = c D_{\vec{u}}[v\circ T^n](\omega)    
    \end{equation}
    Simplifying $J_{v\circ T^n}$ we find
    \begin{align}
        J_{v\circ T^n}(\omega) &= J_v(T^n\omega)J_{T^n}(\omega) \label{i.j}\\
        &= J_v(T^n\omega)J_{T}(T^{n-1}\omega)J_{T^{n-1}}(\omega) \nonumber \\
        &\vdots \nonumber \\ 
        &=  J_v(T^n\omega) \prod_{k=0}^{n-1}J_T(T^k\omega)\nonumber \\
        &= J_v(T^n\omega) \begin{pmatrix}
            2 & 1\\
            1 & 1
        \end{pmatrix}^n \nonumber
    \end{align}
    Combining (\ref{i.lb2}) and (\ref{i.j})
    \[D_{\vec{u}}[v\circ T^n](\omega) = J_v(T^n\omega) \begin{pmatrix}
            2 & 1\\
            1 & 1
        \end{pmatrix}^n \vec{u} = J_v(T^n\omega) \alpha^n \vec{u} = \alpha^n D_{\vec{u}}[v](T^n\omega) > c\alpha^n\]
    which is (\ref{i.lb}).
\end{proof}

\begin{lemma}\label{l.disc}
    Let $\D'_{n,z}, \D_{n,z}$ be the set of discontinuities for $\theta(T^nW_z)$ and $\theta_n|_{W_z}$, respectively. Let $Card(S)$ be the cardinality of a set $S$, then
    \begin{align}
        Card(\D'_n) < \alpha^{n+2} \label{e.disc}\\
        Card(\D_n) < \sum_{j=0}^{n+1} \alpha^{j+1} \label{e.disc2}
    \end{align}
    for any $n\in \ZZ_{\geq 0}$ and $z\in (0,\sqrt{2})$.
\end{lemma}
\begin{proof}
    The discontinuities of 
    \[\theta(T^n\omega) = \cot^{-1}(t-v(T^n\omega))\]
    restricted to $W_z$ is the set of discontinuities for $v\circ T^n$. By our assumptions on $v$, the discontinuities occur when $T^nW_z$ pass through the boundary of the fundamental domain. Concretely, consider any $n\in \NN$ and recall $\ell_z\in [0, \sqrt{\frac{5-\sqrt{5}}{2}}]$ is the length of $W_z$. The $T^{n}$ action on $W_z$ moves this line to another parallel line and stretches the line by a factor of $\alpha^{n}$. Hence, the length of $T^{n}W_z$ is $\alpha^{n}\ell_z$.
    
    With $\alpha^n\ell_z$ at hand, one can compute the maximal number of $\ZZ^2$ crossings. To do so, remark that $\alpha^{n}\ell_z$ can be thought of as the hypotenuse of a triangle. Therefore the lengths of the width and height is bounded above by $\alpha^{n}\ell_z$. This allows for us to count the number of $\ZZ$ crossings in the $x$ and $y$ direction. Counting in this way provides the bounds $2\ell_z\alpha^n + 2 < \alpha^{n+2}$. Consequently,
    \[Card(D'_{n,z}) < \alpha^{n+2}\]
    which shows (\ref{e.disc}).
    
    To show (\ref{e.disc2}), consider $n=0$ and notice that $\D_0 = \D^1_0$, since $\theta_0 = \theta$. Assume that (\ref{e.disc2}) holds for some $n\in \NN$ and observe that $g(T^{n+1}W_z)$ and $\theta(T^{n+1}W_z)$ have the same set of discontinuities $\D'_{n+1,z}$. Then
    \[\phi_{n+1} = \cot^{-1}(\lambda^2 g(T^{n+1}\omega)\cot\theta_n(\omega))\]
    has discontinuities $\D_{n,z}$ as well as $\D'_{n+1,z}$. Hence, the set of discontinuities of $\theta_{n+1}|_{W_z}$ satisfies the containment $\D_{n+1,z} \subset \D'_{n,z}\cup \D_{n,z}$. By induction (\ref{e.disc2}) holds for any $n\in \ZZ_{\geq 0}$.
\end{proof}

\begin{remark}\label{r.1}
    With Lemma \ref{l.lb} and \ref{l.disc} we have that $\theta_n|_{W_z}$ is piecewise $C^1$ and monotone. Moreover, in the proof of Lemma \ref{l.disc}, we can additionally add two more points to $\D'_{n,z}$, since $2\ell_z\alpha^n + 4 < \alpha^{n+2}$ for every $n\in \ZZ_{\geq 0}$. Thus, the bound in both (\ref{e.disc}) and (\ref{e.disc2}) holds with the additional points. The importance is that in Lemma \ref{l.card}, we will need to count the discontinuities plus exactly two more points. This means that the two extra points are already accounted for, so we do not have to add additional points.    
\end{remark}

\begin{lemma}\label{l.card}
    For each $z\in (0,\sqrt{2})$, we define
    \[\C_{n,z} = \{\omega\in W_z: \theta_n(\omega)\in \frac{\pi}{2} + \pi\ZZ\}.\]
    Then
    \begin{equation}\label{e.card}
        Card(\C_{n,z}) < \sum_{j=0}^{n+2} \alpha^{j+1}
    \end{equation}
    for all $n\in \NN$ and $z\in (0,\sqrt{2})$.
\end{lemma}
\begin{proof}
    We induct over $n$. Let $n=0$ and notice that \[\theta_0(\omega) = \frac{\pi}{2} \iff t = v(\omega)\]
    This can only occur at one $\omega\in W_z$, since we have that $D_{\vec{u}}v > c$. Now assume for some $n$ that (\ref{e.card}) holds. Observe that
    \[\{\omega\in W_z: \phi_{n+1}(\omega)\in \frac{\pi}{2} + \pi\ZZ\} = \{\omega\in W_z: \theta_{n}(\omega)\in \frac{\pi}{2} + \pi\ZZ\}\]
    and by Lemma \ref{l.disc} we have $k<\sum_{j=0}^{n+2} \alpha^{j+1}$ discontinuities for $\theta_{n+1}|_{W_z}$. By remark \ref{r.1}, we may assume that $k$ accounts for the two additional points. Therefore, we may label the discontinuities and two additional points, which are the end points, in order: 
    \[\omega_1 < \omega_2 < \cdots < \omega_k.\]
    The order is induced by the homeomorphism of $W_z$ to $[0, \ell_z)$. For each $1\leq j \leq k - 1$ we define $I_j = [\omega_j, \omega_{j+1})$. We make two remarks
    \begin{enumerate}
        \item $W_z = \bigsqcup_{j=1}^{k-1} I_j$.
        \item For each $1< j< k - 1$, then $T^{n+1}I_j = W_l$ for some $l\in (0, \sqrt{2})$.  It may fail for $j=1$ and $k-1$, since it is possible the end points do not stretch to form an entire $W_l$ leaf. Though, the end intervals are subsets of some $W_l$.
    \end{enumerate}
    Consider any $I_j$,
    \begin{align*}
        |\theta_{n+1}(I_j)| &\leq |\phi_{n+1}(I_j)| + |\theta(T^{n+1}I_j)|\\
        &\leq |\phi_{n+1}(I_j)| + |\theta(W_l)|\\
        &< |\phi_{n+1}(I_j)| + \pi
    \end{align*}
    by Lemma \ref{l.disc} we have that $\theta_{n+1}$ restricted to leaves are $C^1$ and monotone. In other words, $\theta_{n+1}$ is $C^1$ and monotone on $I_j$ and for $\omega\in W_z$ we have $c < \theta(\omega) < \pi - c$. These facts imply
    \[Card(\C_{n+1}\cap I_j) \leq Card(C_n\cap I_j) + 1.\]
    By direct computation 
    \begin{align*}
        Card(\C_{n+1}) &= \sum_{j=1}^{k-1} Card(\C_{n+1}\cap I_j)\\
        &\leq \sum_{j=1}^{k-1} Card(\C_{n}\cap I_j) + k - 1\\
        &\leq Card(\C_n) + k\\
        &< \sum_{j=0}^{n+2} \alpha^{j+1} + \sum_{j=0}^{n+2} \alpha^{j+1}\\
        &= 2\sum_{j=0}^{n+2} \alpha^{j+1}\\
        &\leq \sum_{j=0}^{n+3} \alpha^{j+1}
    \end{align*}
    we have shown Lemma \ref{l.card}.
\end{proof}

As a consequence of the Lemmas, one can now bound the measure of the set of $\delta$-balls around \say{bad points}.

\begin{corollary}\label{c.bm}
    Let $\delta>0$ and $||\cdot||_{\RR\PP^1}$ represent the distance to the closest $\pi\ZZ$ point. Define
    \[B_{n,z}(\delta):= \{\omega\in W_z: ||\theta_n(\omega)-\frac{\pi}{2}||_{\RR\PP^1} < \delta\}\]
    then 
    \[\mu_{W_z}(B_{n,z}(\delta)) < C\delta\]
    for all $n\in \NN$ and $z\in (0,\sqrt{2})$.
\end{corollary}
\begin{proof}
    By Lemma \ref{l.card}, we have $\sum_{j=0}^{n} \alpha^{j+1}$ many bad points. By Lemma \ref{l.lb} we must move through the $\delta$-balls at the speed of $c\alpha^n$. Combining these facts with $\alpha^{-1}< 1$ we find
    \[\mu_W(B_n(\delta)) \leq  \frac{2\delta}{\alpha^n}\sum_{j=0}^{n+1} \alpha^{j+1} = (\alpha + \alpha^2 + \sum_{j=0}^{n-1} \alpha^{- j}) 2\delta < 4\alpha^2\delta < C\delta.\]
\end{proof}

We are now ready to prove proposition \ref{p.fPLE}. The argument follows an identical structure as \cite{zhenghe2}, which followed by similar arguments in \cite{young}. For completeness we include the proof.

\begin{proof}[Proof of Proposition \ref{p.fPLE}]
    Let $\vec{e}_1, \vec{e}_2$ represent the standard basis of $\RR^2$, then observe 
    \[\lim_{n\rightarrow \infty} \frac{1}{n}\int_{W_z} \log||A_n(\omega)|| dm(\omega) \geq \lim_{n\rightarrow \infty} \frac{1}{n}\int_{W_z} \log||A_n(\omega)\vec{e}_1|| dm(\omega)\]
    and existence follows from Osceledets Multiplicative Ergodic Theorem. Let $\omega\in W_z$, then define
    \[\vec{w}_n(\omega) = A_n(\omega)\vec{e}_1 \text{ and } \vec{v}_n(\omega) = R_{\theta(T^n\omega)}A_n(\omega)\vec{e}_1.\]
    Rewriting we find
    \[\vec{w}_n(\omega) = \Lambda(T^{n+1}\omega) \vec{v}_{n-1}(\omega) \text{ and } \vec{v}_n(\omega) = R_{\theta(T^n\omega)}\vec{w}_n(\omega).\]
    The previous line implies
    \[||\vec{v}_n(\omega)|| = ||\vec{w}_n(\omega)||.\]
    By definition of $\theta_n(\omega)$ we find
    \[\vec{v}_n(\omega) = \cos\theta_n(\omega)||\vec{v}_n(\omega)||\vec{e}_1 + \sin\theta_n(\omega)||\vec{v}_n(\omega)||\vec{e}_2.\]
    Inductively we find
    \begin{align*}
        ||\vec{w}_n(\omega)|| &= ||\Lambda(T^{n+1}\omega) \vec{v}_{n-1}||\\
        &\geq ||\Lambda(T^{n+1}\omega)\cos\theta_{n-1}(\omega)||\vec{v}_{n-1}(\omega)||\vec{e}_1||\\
        &\geq \lambda|\cos\theta_{n-1}(\omega)| ||\vec{v}_{n-1}(\omega)||\\
        &= \lambda|\cos\theta_{n-1}(\omega)| ||\vec{w}_{n-1}(\omega)||\\
        &\vdots\\
        &\geq \lambda^n \prod_{j=0}^{n-1}|\cos\theta_j(\omega)|.
    \end{align*}
    Applying $\log$ on both sides we have 
    \[\log||A_n(\omega)\vec{e}_1|| = \log||\vec{w}_n(\omega)||\geq \log(\lambda^n \prod_{j=0}^{n-1}|\cos\theta_j(\omega)|) = n\log \lambda + \sum_{j=0}^{n-1} \log|\cos\theta_j(\omega)| \]
    and it follows that
    \begin{equation}\label{boundthis}
        \lim_{n\rightarrow \infty} \int_{W_z} \frac{1}{n}\log ||A_n(\omega)\vec{e}_1|| m(\omega)\geq  \log \lambda + \limsup_{n\rightarrow \infty}\frac{1}{n}\sum_{j=0}^{n-1} \int_{W_z} \log|\cos\theta_j(\omega)| dm(\omega).
    \end{equation}
    From (\ref{boundthis}), it suffices to find a lower bound of $\int_{\TT^2} \log|\cos\theta_j(\omega)| dm(\omega)$.
    Fix any $k \in \{0,\dots, n-1\}$, the estimate will be independent of $k$, and define for all $i\in \NN$
    \[J_i := \{\omega\in W_z: ||\theta_k(\omega) - \frac{\pi}{2}||_{\RR\PP^1}< \frac{\delta}{2^i}\}.\]
    By Corollary \ref{c.bm}  
    \[m(J_i) < \frac{C\delta}{2^i}.\]
    A computation gives
    \begin{align*}
        \int_{J_0} \log|\cos\theta_k(\omega)|dm(\omega) &= \int_{J_0} \log|\sin(\theta_k(\omega) - \frac{\pi}{2})|dm(\omega) \\
        &\geq \int_{J_0} \log\frac{2}{\pi}|\theta_k(\omega) - \frac{\pi}{2}|dm(\omega) \\
        &= \sum_{i\in \NN} \int_{J_i\setminus J_{i+1}} \log\frac{2}{\pi}|\theta_k(\omega) - \frac{\pi}{2}|dm(\omega)\\
        &\geq \sum_{i\in \NN} \int_{J_i\setminus J_{i+1}} \log(\frac{2}{\pi} \frac{\delta}{2^{i+1}})dm(\omega)\\
        &= \sum_{i\in \NN} m(J_i\setminus J_{i+1}) \log(\frac{\delta}{\pi2^i})\\ 
        &\geq \sum_{i\in \NN}  \frac{C\delta i}{2^i}\log((\frac{\delta}{\pi})^{\frac{1}{i}} \frac{1}{2})\\
        &\geq \sum_{i\in \NN}  -\frac{C\delta i}{2^i}\\
        &\geq -C\delta.
    \end{align*}
    Consider
    \[J_0^c = \{\omega\in W_z: ||\theta_k(\omega)-\frac{\pi}{2}||\geq  \delta\}\]
    and it follows
    \begin{align*}
        \int_{J_0^c} \log|\cos\theta_k(\omega)|dm(\omega) &= \int_{J_0^c} \log|\sin(\theta_k(\omega) - \frac{\pi}{2})|dm(\omega)\\
        &\geq \int_{J_0} \log\frac{2}{\pi}|\theta_k(\omega) - \frac{\pi}{2}|dm(\omega)\\
        &\geq \int_{J_0} \log(\frac{2}{\pi}\delta)dm(\omega)\\
        &> C\log\delta.
    \end{align*}
    Picking $\delta = \frac{1}{3}$ and putting everything together
    \begin{align*}
        \int_{W} \log|\cos\theta_k(\omega)|dm(\omega) &= \int_{J_0^c} \log|\cos\theta_k(\omega)|dm(\omega) + \int_{J_0} \log|\cos\theta_k(\omega)|dm(\omega)\\
        &> -\frac{1}{3}C - C\log3\\
        &> -C
    \end{align*}
    Therefore we have
    \[\frac{1}{n}\sum_{j=0}^{n-1} \int_{W_z} \log|\cos\theta_j(\omega)| dm(\omega) > -C\]
    for all $n\in \NN$. Implying
    \[\lim_{n\rightarrow \infty} \frac{1}{n}\int_{W_z} \log ||A_n(\omega)\vec{e}_1|| dm(\omega) > \log \lambda - C_0.\]
    Completing the proof of Proposition \ref{p.fPLE}, which proves Theorem \ref{t.plepd} and hence Theorem \ref{t.PLE}.
\end{proof}

\appendix

\section{Polar Decomposition of Schr\"odinger Cocycle}\label{apd}
Given a matrix $M\in \mathrm{SL}(2,\RR)$ 
\begin{enumerate}
    \item \textbf{Polar decomposition} allows us to write $M$ as $M = S_1\sqrt{M^t M}$ where $S_1\in \mathrm{SO}(2,\RR)$ and $\sqrt{M^t M}$ is positive symmetric.
    \item \textbf{Singular value decomposition} allows us to write $\sqrt{M^t M} = S_2\Lambda S_2^t$ where $S_2\in \mathrm{SO}(2,\RR)$ and $\Lambda$ is precisely $\begin{pmatrix}
        ||M|| & 0\\
        0 & ||M||^{-1}
    \end{pmatrix}$.
\end{enumerate}
Applying polar decomposition and singular value decomposition on $M$ one finds
\[M = S_1S_2\Lambda S_2^t.\]
Moreover, for $M\in C^r(\RR/\ZZ, \mathrm{SL}(2,\RR))$, $r\geq 1$, we can let $S(x) = S_1(x)S_2(x)$ and $O(x)= S_2^t(Tx) (S_1S_2)(x)$, which gives
\[S^{-1}(Tx)M(Tx)S(x) = \Lambda(Tx)O(x)\]
By Lemma 10 in \cite{zhenghe1}, $S_1,S_2, \Lambda$ are also $C^1$, as long as $M(x)\notin\mathrm{SO}(2)$ for all $x$. This implies that $M, (\Lambda\circ T)\cdot O$ are conjugate and hence
\[L(T,M) = L(T, (\Lambda\circ T)\cdot O).\]
This allows for one to study the Lyapunov exponent for the Schr\"odinger Cocycle of $\Lambda(Tx)O(x)$ instead of $M$.

Let us explicitly compute this for the Schr\"odinger cocycle. Before doing so recall we may assume that $||v||_\infty \leq 1$ and
\[A^{(E-\lambda v)} = \begin{pmatrix}
    E - \lambda v & -1\\
    1 &0
\end{pmatrix}.\]
We would like to move to dependence on $\lambda$, so we conjugate our cocycle with
\[P = \begin{pmatrix}
    \frac{1}{\sqrt{\lambda}} & 0\\
    0 & \sqrt{\lambda}
\end{pmatrix}.\]
A computation allows us to define $A$ as
\[A = PA^{(E-\lambda v)}P^{-1} = \begin{pmatrix}
    \lambda r & -\frac{1}{\lambda}\\
    \lambda & 0
\end{pmatrix}\]
where $t = \frac{E}{\lambda}$. By the assumption on $v$, we may consider $t\in \I = [-1, 2]$. Define $r: \TT^2\times \I \rightarrow \RR$
\[r(\omega, t) = t-v(\omega)\]
and note that $r\in C^1([0,1)^2 \times \I, \RR)$. In the process of Polar Decomposition, we find an eigenvalue
\[\alpha = \frac{\lambda^2}{2}\Big[r^2 + 1 + \frac{1}{\lambda^4} + \sqrt{(r^2 + 1 + \frac{1}{\lambda^4})^2 - \frac{4}{\lambda^4}}\Big]\]
for $A^TA$. Hyperbolicity of $A\in \mathrm{SL}(2,\RR)$ implies the eigenvalues are $\alpha, \alpha^{-1}$. Define
\[\beta := r^2 + 1 + \frac{1}{\lambda^4}  \sqrt{(r^2 + 1 + \frac{1}{\lambda^4})^2 - \frac{4}{\lambda^4}}.\]
The singular values are $\sqrt{\alpha}, \frac{1}{\sqrt{\alpha}}$ and hence $||A|| = \lambda\sqrt{\frac{\beta}{2}}$. By normalizing and letting \[f(\omega) = \frac{1}{\sqrt{(\beta - \frac{2}{\lambda^4})^2 + (\frac{2r}{\lambda^2})^2}}\]
we can define
\[S_2 = \frac{1}{\sqrt{(\beta - \frac{2}{\lambda^4})^2 + (\frac{2r}{\lambda^2})^2}}\begin{pmatrix}
    \beta - \frac{2}{\lambda^4} & \frac{2r}{\lambda^2}\\
    -\frac{2r}{\lambda^2} & \beta - \frac{2}{\lambda^4}
\end{pmatrix} = f(x,y)\begin{pmatrix}
    \beta - \frac{2}{\lambda^4} & \frac{2r}{\lambda^2}\\
    -\frac{2r}{\lambda^2} & \beta - \frac{2}{\lambda^4}
\end{pmatrix}.\]
Defining
\[\kappa:= \sqrt{\frac{2}{\beta(\omega)}}\beta(T\omega)\beta(\omega)f(T\omega)f(\omega)\]
we can compute the upper left entry of $O(x)$.
\[O_{11}(\omega, t, \lambda) = \kappa\big[r(\omega) - \frac{2r(\omega)}{\lambda^4\beta(T\omega)} - \frac{2r(T\omega)}{\lambda^2\beta(T\omega)} + \frac{4r(T\omega)}{\lambda^6\beta(T\omega)\beta(\omega)}\big]\]
If we let $\lambda\rightarrow \infty$ we find
\begin{align*}
    \beta(\omega,t, \lambda) &\rightarrow 2r^2+2\\
    f(\omega, t, \lambda) &\rightarrow \frac{1}{2r^2+2}
\end{align*}
and it follows that
\[O_{11}(\omega, t, \lambda) \rightarrow \frac{r(\omega)}{\sqrt{r^2(\omega)+1}}\]
where convergence is thought of with respect to $C^1([0,1)^2\times \I, \RR)$. Defining $g: [0,1)^2\times \I \rightarrow \RR$
\[g(\omega, t) = r^2 + 1\]
then $g\in C^1([0,1)^2\times \I, \RR)$. For sufficiently large $\lambda$, $O_{11}(\omega, t, \lambda)$ can be replaced with $O_{11}(\omega, t, \infty)$ since they are close in the $C^1$ norm. Therefore, we may consider the coycle of the form
\[A(\omega, t, \lambda):= \begin{pmatrix}
    \lambda \sqrt{g(T\omega, t)} & 0 \\
    0 & \frac{1}{\lambda} \sqrt{\frac{1}{g(T\omega, t)}}
\end{pmatrix}R_{\theta(\omega, t)}.\]

\end{document}